\documentclass[a4paper,12pt]{article}
\usepackage{amsmath, amsthm, amsfonts, amssymb}

\newcommand{\supp}{\mathop{\rm supp}}
\renewcommand{\lg}{\langle}
\newcommand{\rg}{\rangle}
\makeatletter
\@addtoreset{equation}{section}
\makeatother
\makeatletter
\@addtoreset{corl}{section}
\makeatother
\makeatletter
\@addtoreset{expl}{section}
\makeatother
\makeatletter
\@addtoreset{thm}{section}
\makeatother
\makeatletter
\@addtoreset{lem}{section}
\makeatother
\makeatletter
\@addtoreset{defn}{section}
\makeatother

\theoremstyle{plain}
\newtheorem{thm}{Theorem}[section]
\newtheorem{lem}{Lemma}[section]

\newtheorem{corl}{Corollary}[section]
\theoremstyle{definition}
\newtheorem{defn}{Definition}[section]
\newtheorem{expl}{Example}[section]
\theoremstyle{remark}

\newcommand{\wt}{\widetilde}

\newcommand{\ov}{\overline}

\newcommand{\ve}{\varepsilon}
\newcommand{\vf}{\varphi}

\newcommand{\mbR}{{\mathbb R}}

\newcommand{\mbQ}{{\mathbb Q}}

\newcommand{\mfX}{{\mathfrak X}}

\renewcommand{\ln}{\log}

\newcommand{\cF}{{\cal F}}

\newcommand{\cY}{{\cal Y}}
\newcommand{\cM}{{\cal M}}
\newcommand{\cR}{{\cal R}}

\begin{document}
 \large

\renewcommand{\proofname}{Proof}
\begin{center}
{\Large\bf
Large deviations for  flows of interacting Brownian motions
}\\[1cm]
A.A.Dorogovtsev, O.V.Ostapenko\\[0,4cm]
{\bf Abstract}
\end{center}
\vskip15pt

 We establish  the large deviation   principle (LDP) for
stochastic flows of interacting Brownian motions. In particular,
we consider smoothly correlated flows, coalescing flows and Brownian  motion stopped at a hitting moment.
 \vskip0,5 cm



\section{Introduction}
\label{section1}

The article is devoted to the large deviations principle  for
stochastic flows of Brownian motions on $\mbR.$  We use the
following definition.

\begin{defn}
\label{defn1.1} A random field $\{x(u,t); u\in\mbR, t\geq0\}$ is a
stochastic flow of Brownian motions if the following conditions
hold:

1) for every $u\in\mbR \ x(u,\cdot)$ is a Wiener martingale with
respect to a common filtration and $x(u,0)=u,$

2) for every $u_1\leq u_2$  and $t\geq0$
$$
x(u_1,t)\leq x(u_2,t).
$$
\end{defn}

Such stochastic flows can be constructed via different ways [1 --
3]. Here we study flows produced by solutions of stochastic differential equations, and their limits. Let $W$ be a
Wiener sheet on $\mbR\times[0; +\infty)$   (i.e. $W$ is the Gaussian
random measure with independent values on disjoint sets and the
Lebesgue measure as a control measure). Consider the equation
\begin{equation}
\label{eq:1.1}
\begin{cases}
dx(u,t)=\int_\mbR\vf(x(u, t)-p)W(dp, dt),\\
x(u,0)=u, \ u\in\mbR
\end{cases}
\end{equation}
with a smooth function $\vf\in S$  ($S$ denotes the Schwartz space). Suppose,
that
$$
\int_{\mbR}\vf^2(p)dp=1.
$$
Using an orthonormal basis $\{e_n; n\geq1\}$ in $L_2(\mbR),$ (1.1)
can be rewritten as a SDE with respect to the countable family of
independent Wiener processes
$$
\beta_n(t)=\int^t_0\int_{\mbR}e_n(p)W(dp, ds), \ n\geq1.
$$
In such terms (1.1) takes the form
\begin{equation}
\label{eq1.1'}
\begin{cases}
dx(u,t)=\sum^\infty_{n=1}a_n(x(u,t))d\beta_n(t),\\
x(u,0)=u, \ u\in\mbR.
\end{cases}
\end{equation}

Under our conditions on $\vf$ Equation (1.1)  has a unique
solution, which is a flow of diffeomorphisms  [3, 5]. Consider the
properties of $x.$ Note that for arbitrary $u_1\leq u_2$  and
$t\geq0$
$$
x(u_1, t)\leq x(u_2,t)
$$
(diffeomorphic property). For every $u$ \ $x(u, \cdot)$ is a
continuous martingale with the characteristics
$$
\lg x(u, \cdot)\rg_t=\int^t_0\int_{\mbR}\vf^2(x(u,s)-p)dpds=t, \ t\geq0.
$$
Consequently [4],  $x(u, \cdot)$ is a Brownian motion, starting from
$u.$  Note that for different points $u_1, u_2$  the processes
$x(u_1, \cdot)$ and $x(u_2, \cdot)$ are correlated:
$$
\lg x(u_1, \cdot), x(u_2, \cdot)\rg_t=
$$
$$
=\int^t_0\int_{\mbR}\vf(x(u_1, s)-p)\vf(x(u_2, s)-p)dpds=
$$
$$
=\int^t_0\Phi(x(u_1, s)-x(u_2, s))ds,
$$
where
$$
\Phi(r)=\int_{\mbR}\vf(r-p)\vf(p)dp.
$$
The function $\Phi$  can be treated as a momental correlation
between different one-point motions in the flow \eqref{eq:1.1}
[3]. The singular case can be obtained from \eqref{eq:1.1}  when
$\vf^2$ tends to $\delta_0.$  More precisely, the following result
was proved in [5].

\begin{thm}
\label{thm:1.1}  Let $\vf_\ve\in S, \ \supp\vf_\ve\subset[-\ve;
\ve],$
$$
\int_{\mbR}\vf^2_\ve(p)dp=1,
$$
$\vf^2_\ve\to\delta_0, \ \ve\to0+.$  Then the $n$-pont motions of
the flow \eqref{eq:1.1} with $\vf_\ve$ instead of $\vf$  converge in
distribution to the $n$-point motions of Arratia's flow.
\end{thm}

Recall, that Arratia's flow consists of Brownian motions, which are
independent up to their meeting and move together after that
[2]. This flow can be treated as a Brownian flow with
$\delta_0$-correlation.

A recent result concerning the LDP for stochastic flows of diffeomorphisms is due to A. Budhiraja, D. Dupuis and V. Maroulas [6]. These authors considered stochastic dynamical systems driven by an infinite-dimensional Brownian motion
$$
\begin{aligned}
&d\phi^\ve_{s,t}(x)=F^\ve(\phi^\ve_{s,t}(x), dt),\\
&
\phi^\ve_{s,s}(x)=x, \ 0\leq s\leq t\leq T, \ \ x\in\mbR^d,
\end{aligned}
$$
where $F^\ve(x, t)$   is a $C^{k+1}$-Brownian motion in the sence of
Kunita [3]. They got the LDP for $(\phi^\ve, F^\ve)_{\ve>0}$ in the space
$C([0,T]; C^m(\mbR^d))$ and $C([0,T]; G^m)$ ($G^m$ is the group of
$C^m$-diffeomorphisms on $\mbR^d$). The proof of the LDP is based on variational representations for functionals of infinite-dimensional Brownian motion [8].

Note that the family $\{x^\ve\}_{\ve>0}$ in (1.1) can be represented as the flow described in [6]. Indeed, we can
write

$$
dx^\ve(u,t)=\sqrt{\ve}\sum^\infty_{k=1}a_k(x^\ve(u,t))d\beta_k(t),
$$
$$
x^\ve(u,0)=u, \ t\in[0;1], \ u\in\mbR,
$$
where
$$
a_k(v)=\int_{\mbR}\vf(v-p)e_k(p)dp,
$$
$$
\beta_k(t)=\int^t_0\int_{\mbR}e_k(p)W(dp, ds),
$$
$\{e_k, k\geq1\}$   is an orthonormal basis in $L_2(\mbR).$   However, in our article we present the LDP not only for smoothly correlated, but also for coalescing flows. In particular, we prove the LDP for Arratia's flow. The main dif\-fe\-ren\-ce between
 smooth and singular correlation is the presence of $W$ in
\eqref{eq:1.1}. Really, traditionally one can get the LDP for
\eqref{eq:1.1} from $W$ (roughly speaking). But
Arratia's flow does not contain a white noise. Hence we will need an
additional construction. The article is organized as follows. In the
second part we prove the LDP for smooth case \eqref{eq:1.1} using a method different from that of [6]. The
third part is devoted to the structure of Arratia's flow, which is a
limit case of (1.1). Here we present some facts about the total time
of free motion for particles in this flow. The fourth part is devoted
to the LDP for Brownian motion stopped at a hitting moment. Here
the time scaling is used instead of the space scaling. The fifth part
contains the LDP for the $n$-point motions of Arratia's flow. Finally
the last part of the article deals with the LDP for Arratia's flow in
the L\'evy--Prokhorov distance.

\section{LDP for smoothly correlated flows}
\label{section2}

For $\ve>0$ let  $x^\ve$  be the stochastic flow  described by the
dif\-fe\-rential equation
\begin{equation}
\label{eq:2.1}
\begin{array}{l}
dx^\ve(u, t)=\sqrt{\ve}\int_{\mbR}\vf(x^\ve(u,t)-p)W(dp, dt),\\
x^\ve(u,0)=u, \ u\in\mbR, \ t\in[0, 1].
\end{array}
\end{equation}

For every $\ve>0$ \ $x^\ve$  is a flow of homeomorphisms in $\mbR$
[3]. Moreover, one can get the following relations describing the
growth of $x^\ve$   with respect to spatial variable
\begin{equation}
\label{eq:2.2} \forall \  \delta>0: \ \  \lim_{|u|\to+\infty}
\frac{|x^\ve(u, t)|}{1+|u|^{\delta+1}}= \lim_{|u|\to+\infty}
\frac{|u|}{1+|x^\ve(u, t)|^{\delta+1}}=0 \ \mbox{a.s.}
\end{equation}
It follows from this property, that $x^\ve$  can be considered as a random
element in the space $\mfX=C([0; 1]; L_2(\mbR, \mu)).$   Here $\mu$  is the
standard normal distribution on $\mbR.$  We suppose, that $\mfX$  is
equipped with the norm
\begin{equation}
\label{eq:2.3}
\mfX\ni x\mapsto \|x\|=\sup_{t\in[0; 1]}
\bigg(
\int_{\mbR} x(u, t)^2\mu(du)
\bigg)^{\frac{1}{2}}.
\end{equation}
 We are
going to establish the  LDP for $x^\ve$  in the space $\mfX.$  The main
result is based on an extension of the contraction principle to
maps that are not continuous, but can be approximated well by
continuous maps [9]. Define an approximation for $x^\ve$  as
follows. Let $x^\ve_m, m\geq1$   be the stochastic flow described by
the equation
\begin{equation}
\label{eq:2.4}
\begin {cases}
dx^\ve_m(u, t)=\sqrt{\ve}\int_{\mbR}\vf\bigg(
x^\ve_m\bigg(u, \frac{[tm]}{m}\bigg)-p\bigg)W(dp, dt)\\
x^\ve_m(u,0)=u, \ u\in\mbR, \ t\in[0;1],
\end{cases}
\end{equation}
in which the coefficients of \eqref{eq:2.1}  are frozen over the time
intervals
$\bigg[\frac{k}{m}, \frac{k+1}{m}\bigg), \ k=0, \ldots, m-1.$
  Consider first the case $m=1.$
In this case \eqref{eq:2.4}  has the form
\begin{equation}
\label{eq:2.5}
\begin {cases}
dy^\ve(u, t)=\sqrt{\ve}\int_{\mbR}\vf(u-p)W(dp, dt),\\
y^\ve(u,0)=u, \ u\in\mbR, \ t\in[0;1].
\end{cases}
\end{equation}
Note that $y^\ve$ is a Gaussian random element in $\mfX$ obtained
from $W.$   So the LDP  for the family $y^\ve$ has a known form [10,
11]. Define  $H$  as the set of functions of the type
$$
h(u, t)=u+\int^t_0\int_{\mbR}\vf(u-p)a(p, s)dp ds,
$$
$$
 a\in
L_2(\mbR\times[0; 1]), \ u\in\mbR, \ t\in[0;1].
$$
Also, let $\cF$  be the Fourier transform on $\mbR.$
\begin{thm}
\label{thm:2.1}  Let $\vf\in S$  be such, that $\cF(\vf)\ne0$ a.s.
Then the family $\{y^\ve\}$   satisfies the LDP in $\mfX$  with rate
function
\begin{equation}
\label{eq:2.6} I_y(h)=\begin{cases}
\frac{1}{4\pi}\int^1_0\int_{\mbR} \bigg( \frac{\cF(\dot{h}(
\cdot,t))(\lambda)}{\cF(\vf)(\lambda)}
\bigg)^2d\lambda dt, \ h\in H,\\
\infty, \ h\notin H,
\end{cases}
\end{equation}
i.e.:

1) for each closed set $F\subset \mfX$
$$
\mathop{\varlimsup}\limits_{\ve\to0}\ve\log
P\{y^\ve\in F\}\leq-\inf_{f\in F}I(f);
$$

2) for each open set $G\subset\mfX$
$$
\mathop{\varliminf}\limits_{\ve\to0}\ve\log P\{y^\ve\in
G\}\geq-\inf_{f\in G}I(f).
$$
\end{thm}

\begin{proof}
Note that
$$
y(u,t)=u+\int^t_0\int_{\mbR}\vf(u-p)W(dp, ds), \ t\in[0, 1], \
u\in\mbR,
$$
is a Gaussian random element in $\mfX.$   Define the map
$i: L_2(\mbR\times[0; 1])\to\mfX$   as
$$
i(a)(u, t)=u+\int^t_0\int_{\mbR}\vf(u-p)a(p, s)dp ds, \ a\in
L_2(\mbR\times[0; 1]).
$$
$i$ is a linear continuous operator.
Note that due to the condition on the Fourier transform of $\vf$
the operator $i$ is injection.

Since $W$   can be considered as a generalized Gaussian random
element in $L_2(\mbR\times[0; 1])$  with zero mean and identity
covariation, then by the standard arguments one can verify, that
covariation of $y$  is $ii^*.$  Therefore
[10,11], $\{y^\ve\}$   satisfies the LDP in $\mfX$  with rate
function
$$
I_y(h)=\begin{cases}
\frac{1}{2}\|i^{-1}(h)\|^2_{L_2(\mbR\times[0;1])}, \ h\in H,\\
\infty, \ h\notin H.
\end{cases}
$$
Rewrite the rate function in terms of the Fourier transform. Let $h=i(a),$
then
$$
\dot{h}(u,t)=\int_{\mbR}\vf(u-p)a(p,t)dp.
$$
For fixed $t\in[0, 1]$ apply the Fourier transform:
$$
\cF(\dot{h}(\cdot,t))=\cF(\vf)\cF(a(\cdot, t)).
$$
Hence,
$$
\|i^{-1}(h)\|^2_{L_2(\mbR\times[0;1])}=
\frac{1}{2\pi}\int^1_0\int_{\mbR} \bigg( \frac{\cF(\dot{h}(
\cdot,t))(\lambda)}{\cF(\vf)(\lambda)} \bigg)^2d\lambda dt.
$$
The theorem is proved.
\end{proof}

The contraction principle
[9] yields the LDP for $\{y^\ve(\cdot,t)\}$ at a fixed time
$t\in[0, 1].$
\begin{corl}
\label{corl:2.1} For any fixed $t\in[0, 1]$ the family $\{y^\ve(
\cdot,t)\}$ satisfies the LDP in ${L}_2(\mbR,\mu)$   with rate
function
\begin{equation}
\label{eq:2.7} I_t(h)=
\begin{cases}
\frac{1}{4\pi t}\int_{\mbR}\bigg(
\frac{\cF(h)(\lambda)}{\cF(\vf)(\lambda)} \bigg)^2d\lambda, \ h\in H_t,\\
\infty, \ h\notin H_t,
\end{cases}
\end{equation}
where $H_t$ is the set of functions from $L_2(\mbR, \mu)$ having square-integrable derivative.
\end{corl}

\noindent {\it Remark}. Observe that $y^\ve$  is a random element in
$\cY=C([0, 1]; {L}_2(\mbR, \nu)),$    where $\nu$  is the measure on
$\mbR$ with the density $\frac{1}{1+|u|^l}, u\in\mbR, l>3.$ And for
any $p\geq2$  and fixed time $t,$   $y^\ve(\cdot,t)$  is a random
element in ${L}_p(\mbR, \mu).$  Using the same arguments as in
Theorem \ref{thm:2.1}  one can show, that $\{y^\ve\}$   satisfies
the LDP in $\cY$  and $\{y^\ve(\cdot,t)\}$   satisfies the LDP in
${L}_p(\mbR, \mu)$   with the same rate functions \eqref{eq:2.6} and
\eqref{eq:2.7}   correspondingly.

To prove the LDP for $x^\ve_m$ and $x^\ve$   we need an
additional assumption on the function $\vf.$  From now on we will
suppose, that the following relation holds:

{\bf H1}.     $ \vf=\psi_1*\psi_2, \ \psi_1, \psi_2\in S. $

 In the case $m>1,$  $x^\ve_m$ can
be represented as
\begin{equation}
\label{eq:2.8}
\begin{array}{l}
x^\ve_m(u,t)=x^\ve_m\bigg(u, \frac{k}{m}\bigg)+\\
+\sqrt{\ve}
\int_{\mbR}\psi_1\bigg(x^\ve_m\bigg(u,\frac{k}{m}\bigg)-q\bigg)
\int^t_{\frac{k}{m}}\int_{\mbR}\psi_2(q-p)W(dp, ds)dq,\\
t\in\bigg[\frac{k}{m}, \frac{k+1}{m}\bigg], k=0,\ldots, m-1.
\end{array}
\end{equation}
By the remark
$\bigg\{\sqrt{\ve}\int^\cdot_{\frac{k}{m}}\int_{\mbR}\psi_2(\cdot-p)W
(dp, ds)\bigg\}$ satisfies the LDP in $\cY$  and
$\bigg\{\sqrt{\ve}\int^{\frac{k+1}{m}}_{\frac{k}{m}}\int_{\mbR}\psi_2
(\cdot-p)W (dp, ds)\bigg\}$ satisfies the LDP in ${L}_p(\mbR, \mu),$
for any $p\geq2.$

Before formulating the main result in this section we will prove some lemmas establishing the continuity property of
some maps.
\begin{lem}
\label{lem:2.1}   For $f\in{L}_{pl}(\mbR,
\mu)$  and $g\in{L}_2(\mbR, \nu), p\geq2,$  define
$$
\cF^p(f, g)(u)=\int_{\mbR}\psi_2(f(u)-q)g(q)dq.
$$
Then $\cF^p$  is continuous map from
$
{L}_{pl}(\mbR, \mu)\times{L}_2(\mbR, \nu)$ to ${L}_p(\mbR, \mu).$
\end{lem}
\begin{proof}  Let $f_n\to f, n\to\infty$   in ${L}_{pl}(\mbR,\mu)$
and $g_n\to g, n\to\infty$  in ${L}_2(\mbR, \nu).$ Then
$$
\|\cF^p(f_n, g_n)-\cF^p(f, g)\|_{{L}_p(\mu)}\leq
$$
$$
\leq \|\cF^p(f_n,
g_n)-\cF^p(f_n,g)\|_{{L}_p(\mu)}+\|\cF^p(f_n,g)-\cF^p(f,g)\|_{
{L}_p(\mu)} \leq
$$
$$
\leq \|g_n-g\|_{{L}_2(\nu)} \bigg[ \int_{\mbR}\bigg(
\int_{\mbR}\psi^2_2(f_n(u)-q)(1+|q|^l)dq \bigg)^{p/2}\mu(du)
\bigg]^{1/p}+
$$
$$
+ \|g\|_{{L}_2(\nu)} \bigg[ \int_{\mbR}\bigg( \int_{\mbR} (
\psi_2(f_n(u)-q)-\psi_2(f(u)-q))^2(1+|q|^l)dq \bigg)^{p/2}\mu(du)
\bigg]^{1/p}.
$$
Observe that for some $C>0$
$$
\int_{\mbR}\psi^2_2(f(u)-q)(1+|q|^l)dq\leq C(|f(u)|^l+1).
$$
Since $f_n\to f, n\to\infty$ in ${L}_{pl}(\mbR, \mu),$  there exists
a constant $C_1>0$ such that
$$\|f_n\|_{{L}_{pl}(\mu)}\leq C_1, \
n\geq1.
$$

   So, for some constant $C_2>0$
$$
\int_{\mbR}\bigg(
\int_{\mbR}\psi^2_2(f_n(u)-q)(1+|q|^l)dq
\bigg)^{p/2}\mu(du)\leq
$$
$$
\leq
\int_{\mbR}C^{p/2}(|f_n(u)|^l+1)^{p/2}\mu(du)\leq
$$
$$
\leq C^{p/2}(\|f_n\|^l_{{L}_{pl}(\mu)}+1)^p\leq C_2, \ n\geq1.
$$
Without loss of generality, we assume that $\{f_n\}$  converges to
$f$   almost everywhere. Denote
$$
b_n(u,q)=(\psi_2(f_n(u)-q)-\psi_2(f(u)-q))^p(1+|q|^l)^{p/2}.
$$
Let us check, that
$$
sup_{n\geq1}
\int_{\mbR}\int_{\mbR}b^2_n(u, q)dq\mu(du)<\infty.
$$
Indeed, for some $C_3>0$
$$
\int_{\mbR}\int_{\mbR}b^2_n(u, q)dq\mu(du)\leq
$$
$$
\leq 2^{2p-1}
\bigg(
\int_{\mbR}
\int_{\mbR}
\psi^{2p}_2(f_n(u)-q)(1+|q|^l)^p dq\mu(du)+
$$
$$
+
\int_{\mbR}
\int_{\mbR}
\psi^{2p}_2(f(u)-q)(1+|q|^l)^p dq\mu(du)\bigg)\leq
$$
$$
\leq 2^{2p-1}C^p \bigg( (\|f_n\|^l_{{L}_{pl}(\mu)}+1)^p +
(\|f\|^l_{{L}_{pl}(\mu)}+1)^p\bigg)\leq C_3, \ n\geq1.
$$
This estimation establishes uniform integrability of $\{b_n\}$  and
proves the lemma.
\end{proof}

\begin{corl}
\label{corl:2.2}   For $f\in{L}_{2l}(\mbR,
\mu)$   and $g\in\cY$  define
$$
\cF(f, g)(t, u)=\int_{\mbR}\psi_2(f(u)-q)g(t, q)dq, \ t\in[0;1], u\in\mbR.
$$
Then $\cF$ is continuous map from ${L}_{2l}(\mbR, \mu)\times\cY$  to
$\mfX.$
\end{corl}

The following representation holds for every $k=1, \ldots, m$
$$
x^\ve_m\left(u,\frac{k}{m}\right)=\cF^{2l} \left(x^\ve_m
\left(\cdot,\frac{k-1}{m}\right), y^\ve
\left(\cdot,\frac{k}{m}\right)-
y^\ve\left(\cdot,\frac{k-1}{m}\right)\right)(u),
$$
$$
x^\ve_m\left(u,\frac{k-1}{m}\right)=\cF^{2l^2} \left(x^\ve_m
\left(\cdot,\frac{k-2}{m}\right), y^\ve
\left(\cdot,\frac{k-1}{m}\right)-
y^\ve\left(\cdot,\frac{k-2}{m}\right)\right)(u), \ldots
$$
$$
x^\ve_m\left(u,\frac{1}{m}\right)=\cF^{2l^k} \left(e,
y^\ve\left(\cdot,\frac{1}{m}\right)\right)(u),
$$
where \begin{equation} \label{2.9'}
y^\ve(u,t)=\sqrt{\ve}\int^t_0\int_{\mbR}\psi_2(u-p)W(dp, ds), \
t\in[0;1], u\in\mbR,
\end{equation}
$$
 e(u)=u, u\in\mbR.
$$
Consequently, $x^\ve_m\left(\cdot,\frac{k}{m}\right)$ is the image of
$y^\ve$  under some continuous map from $\cY$  to ${L}_{2l}(\mbR, \mu).$
By induction one can prove that  $x^\ve_m$  can be represented as the
image of $y^\ve$  under a continuous map from $\cY$ to $\mfX$ and,
consequently, satisfy the LDP. We will show that $\{x^\ve_m\}$ are
exponentially good approximations of $\{x^\ve\},$ i.e. for every
$\delta>0$
$$
\lim_{m\to\infty}\mathop{\varlimsup}\limits_{\ve\to0}\ve\log
P\{\|x^\ve_m-x^\ve\|>\delta\}=-\infty,
$$
where the norm $\|\cdot\|$ is taken in the space $\mfX.$   Following
[9], let us prove some auxiliary results.

Consider the stochastic flow  described by the equation
$$
\begin{array}{l}
dz^\ve(u,t)=\sqrt{\ve}\int_{\mbR}\psi(t,u,p)W(dp, dt),\\
z^\ve(u,0)=\alpha(u), \ u\in\mbR, \ t\in[0;1],
\end{array}
$$
where $\alpha\in{L}_2(\mbR, \mu),$ $\psi\in{L}_2([0,
1]\times\mbR\times\mbR\times\Omega),$ for every $u\in\mbR$  \
$\psi(\cdot, u, \cdot)$ is progressively measurable with respect to the
filtration $\cF_t=\sigma\{w(\Delta), \Delta\subset\mbR\times[0,
t]\}.$ For every $t\in[0;1],$  $z^\ve( \cdot,t)\in{L}_2(\mbR,\mu)$
and $\|z^\ve(\cdot,t)\|^2_{{L}_2(\mu)}$ has the stochastic
differential [3]
$$
d\|z^\ve(\cdot,t)\|^2_{{L}_2(\mu)}=2\sqrt{\ve}\int_{\mbR}\int_{\mbR}
z^\ve(u,t)\psi(t,u,p)\mu(du)W(dp,dt)+
$$
$$
+\ve\int_{\mbR}\int_{\mbR}\psi^2(t,u,p)\mu(du)dpdt.
$$

Let $\tau\in[0;1]$  be a stopping time with respect to the
filtration $\cF_t.$  Suppose, that for some constants ${L}, \rho$
and any $t\in[0, \tau]:$
$$
\int_{\mbR}\int_{\mbR}\psi^2(t,u,p)\mu(du)dp\leq {L}(\|z^\ve(
\cdot,t)\|^2_{{L}_2(\mu)}+\rho^2).
$$
\begin{lem}
\label{lem:2.2} For any $\delta>0,$   $\ve\leq 1$
$$
\ve\log P\{\sup_{t\in[0, \tau]}\|z^\ve(
\cdot,t)\|_{{L}_2(\mu)}\geq\delta\}\leq
{L}+\log\frac{\rho^2+\|\alpha\|^2_{{L}_2(\mu)}}{\rho^2+\delta^2}.
$$
\end{lem}
\begin{proof}
Let $\Phi(t)=f(z^\ve(\cdot,t)),$  where
$$
f(z)=(\rho^2+\|z\|^2_{{L}_2(\mu)})^{1/\ve}.
$$
Using the It\^o formula, we have
$$
d\Phi(t)=\frac{\Phi(t)}{\rho^2+\|z^\ve(\cdot,t)\|^2_{{L}_2(\mu)}}
\bigg[ \int_{\mbR} \int_{\mbR} \psi^2(t,u,p)\mu(du)dp+
$$
$$
+\frac{1}{2}\big(\frac{1}{\ve}-1\big)
\frac{1}{\rho^2+\|z^\ve(\cdot,t)\|^2_{{L}_2(\mu)}} \int_{\mbR}
\bigg( \int_{\mbR} z^\ve(u,t)\psi(t,u,p)\mu(du) \bigg)^2 dp\bigg]dt+
$$
$$
+\frac{2}{\sqrt{\ve}}
\frac{\Phi(t)}{\rho^2+\|z^\ve(\cdot,t)\|^2_{{L}_2(\mu)}} \int_{\mbR}
\int_{\mbR} z^\ve(u,t)\psi(t,u,p)\mu(du)W(dp, dt)=
$$
$$
=\int_{\mbR}f_1(p, t)W(dp, dt)+f_2(t)dt.
$$
Note that for $t\leq\tau$
$$
\int_{\mbR} \bigg( \int_{\mbR} z^\ve(u,t)\psi(t,u,p)\mu(du) \bigg)^2
dp\leq
$$
$$
\leq \|z^\ve(\cdot,t)\|_{{L}_2(\mu)}
{L}(\|z^\ve(\cdot,t)\|_{{L}_2(\mu)}+ \rho^2)\leq
$$
$$
\leq{L}(\|z^\ve(\cdot,t)\|_{{L}_2(\mu)}+\rho^2)^2.
$$
Then
$$
\int_{\mbR}f^2_1(p, t)dp\leq\frac{4{L}\Phi^2(t)}{\ve}
$$
and
$$
f_2(t)\leq \frac{{L} \Phi(t)}{\ve}.
$$
Fix $\delta>0$  and define the stopping time $\tau_1=\inf\{t: \
\|z^\ve(\cdot,t)\|_{{L}_2(\mu)}\geq\delta\}\wedge\tau.$ Since the norm
$\|f_1(t,
\cdot)\|_{L_2(\mbR)}\leq\frac{2\sqrt{{L}}}{\sqrt{\ve}}\Phi(t)$ is
uniformly bounded on $[0, \tau_1],$ it follows that
$\bigg(\Phi(t)-\int^t_0f_2(s)ds\bigg)$ is a continuous martingale up
to $\tau_1.$   Therefore
$$
E\Phi(t\wedge\tau_1)=\Phi(0)+E\int^{t\wedge\tau_1}_0f_2(s)ds \leq
$$
$$
\leq\Phi(0)+\frac{{L}}{\ve}E\int^{t\wedge\tau_1}_0\Phi(s)ds=
\Phi(0)+
\frac{{L}}{\ve}E\int^{t\wedge\tau_1}_0\Phi(s\wedge\tau_1)ds=
$$
$$
=\Phi(0)+ \frac{{L}}{\ve} \int^t_0E\Phi(s\wedge\tau_1)ds.
$$
Consequently, by Gronwall's lemma
$$
E\Phi(\tau_1)=E\Phi(\tau_1\wedge1)\leq\Phi(0)e^{{L}/\ve}.
$$
Therefore, by Chebychev's inequality, we have that
$$
P\{\|z^\ve(\cdot,\tau_1)\|_{{L}_2(\mu)}\geq\delta\}=
 P\{f(\|z^\ve(\cdot,\tau_1)\|_{{L}_2(\mu)})\geq f(\delta)\}\leq
$$
$$
\leq
 \frac{Ef(\|z^\ve(\cdot,\tau_1)\|_{{L}_2(\mu)})}{f(\delta)}
=\frac{E\Phi(\tau_1)}{f(\delta)}.
$$
Since $ \sup_{t\in[0,\tau]}\|z^\ve(\cdot,t)\|_{{L}_2(\mu)}\geq\delta
$ iff $ \|z^\ve(\cdot,\tau_1)\|_{{L}_2(\mu)}\geq\delta, $ then
$$
\ve\log P\{ \sup_{t\in[0,\tau]} \|z^\ve( \cdot,
t)\|_{{L}_2(\mu)}\geq\delta\}=
 \ve\log P\{ \|z^\ve(\cdot,\tau_1)\|_{{L}_2(\mu)}\geq\delta\}\leq
$$
$$
\leq \ve\log \frac {(\rho^2+\|\alpha\|^2_{L_2(\mu)})^{1/\ve}
e^{L/\ve}} {f(\delta)}=
{L}+\log\frac{\rho^2+\|\alpha\|^2_{{L}_2(\mu)}}{\rho^2+\delta^2}.
$$
The lemma is proved.

\end{proof}
\begin{lem}
\label{lem:2.3}
For $\alpha<1/2$ there exists a constant $C>0$  such that for any
$0<\ve<1, m\geq1:$
$$
E\exp \bigg\{\alpha \sup_{t\in\left[\frac{i}{m},
\frac{i+1}{m}\right]} \frac{m}{\ve}\|x^\ve_m(
\cdot,t)-x^\ve_m\left(\cdot, \frac{i}{m}\right)\|^2
_{{L}_2(\mu)}\bigg\}\leq C, \ i=0, \ldots, m-1.
$$
\end{lem}
\begin{proof}
Let us write
$$
\exp \bigg\{\alpha \sup_{t\in\left[\frac{i}{m},
\frac{i+1}{m}\right]} \frac{m}{\ve}\|x^\ve_m(
\cdot,t)-x^\ve_m\left(\cdot, \frac{i}{m}\right)\|^2
_{{L}_2(\mu)}\bigg\}=
$$
$$
=\sum^\infty_{n=0}\frac{\alpha^n}{n!} \bigg(
\sup_{t\in\left[\frac{i}{m}, \frac{i+1}{m}\right]} \frac{m}{\ve}
\|x^\ve_m(\cdot,t)-x^\ve_m\left(\cdot,\frac{i}{m}\right)\|^2
_{{L}_2(\mu)}\bigg)^n.
$$
A general  term of this series can be estimated as
$$
E\bigg( \sup_{t\in\left[\frac{i}{m}, \frac{i+1}{m}\right]}
\|x^\ve_m(\cdot,t)-x^\ve_m\left(\cdot,\frac{i}{m}\right)\|^2
_{{L}_2(\mu)}\bigg)^n\leq
$$
$$
\leq E\ve^n\bigg( \int_{\mbR}  \sup_{t\in\left[\frac{i}{m},
\frac{i+1}{m}\right]} \bigg(\int^t_{\frac{i}{m}}\int_{\mbR}
\vf\bigg(x^\ve_m\bigg(u,\frac{i}{m}\bigg)-p\bigg)W(dp,
ds)\bigg)^2\mu(du) \bigg)^n\leq
$$
$$
\leq \ve^n \int_{\mbR} E \sup_{t\in\left[\frac{i}{m},
\frac{i+1}{m}\right]} \bigg(\int^t_{\frac{i}{m}}\int_{\mbR}
\vf\bigg(x^\ve_m\bigg(u,\frac{i}{m}\bigg)-p\bigg)W(dp,
ds)\bigg)^{2n}\mu(du).
$$
Since $ \int^t_{\frac{i}{m}}\int_{\mbR}
\vf\bigg(x^\ve_m\bigg(u,\frac{i}{m}\bigg)-p\bigg)W(dp, ds) $ is a
continuous martingale, it follows that
[12]
$$
E\bigg( \sup_{t\in\left[\frac{i}{m}, \frac{i+1}{m}\right]}
\|x^\ve_m(\cdot,t)-x^\ve_m\left(\cdot,\frac{i}{m}\right)\|^2
_{{L}_2(\mu)}\bigg)^n\leq
$$
$$
\leq \ve^n\int_{\mbR}\bigg(\frac{2n}{2n-1}\bigg)^{2n} E\bigg(
\int^{\frac{i+1}{m}}_{\frac{i}{m}}\int_{\mbR}
\vf\bigg(x^\ve_m\bigg(u,\frac{i}{m}\bigg)-p\bigg)W(dp,
ds)\bigg)^{2n}\mu(du)=
$$
$$
=\bigg(\frac{\ve}{m}\bigg)^n
\bigg(\frac{2n}{2n-1}\bigg)^{2n}
(2n-1)!!
$$
Consequently for $\alpha<1/2$  the constant $C$   can be chosen to be
$$
C= \sum\limits^\infty_{n=0} \frac{\alpha^n}{n!}
\bigg(\frac{2n}{2n-1}\bigg)^{2n} (2n-1)!!.
$$
The lemma is proved.
\end{proof}

The following lemma shows that $\{x^\ve_m\}$  is exponentially good
approximation of $\{x^\ve\}.$
\begin{lem}
\label{lem:2.4}
For any $\delta>0,$
$$
\lim_{m\to\infty}\mathop{\varlimsup}\limits_{\ve\to0}\ve\log
P\{\|x^\ve-x^\ve_m\|>\delta\}=-\infty.
$$
\end{lem}
\begin{proof}
Fix $\delta>0.$  For any $\rho>0,$ define the stopping time
$$
\tau=\inf\bigg\{ t \ | \ \|x^\ve_m(
\cdot,t)-x^\ve_m\left(\cdot,\frac{[mt]}{m}\right)\|_
{{L}_2(\mu)}\geq\rho\bigg\}\wedge1.
$$
Let $z^\ve(u,t)=x^\ve(u,t)-x^\ve_m(u,t).$  Due to  Lemma
\ref{lem:2.2} for any $\delta>0$  and any $\ve\leq1,$
$$
\ve\log
 P\{\sup_{t\in[0,\tau]}
\|x^\ve(\cdot,t)-x^\ve_m(\cdot,t)\|_{{L}_2(\mu)}>\delta \}\leq
{L}+\log\frac{\rho^2}{\rho^2+\delta^2},
$$
where ${L}$ is  independent of $\ve, \delta, \rho$  and $m.$  Hence,
$$
\lim_{\rho\to0}\sup_{m\geq1}\mathop{\varlimsup}\limits_{\ve\to0}\ve\log
 P
\{\sup_{t\in[0,\tau]} \|x^\ve(
\cdot,t)-x^\ve_m(\cdot,t)\|_{{L}_2(\mu)}>\delta \}=-\infty.
$$
Now, since
$$
\{\|x^\ve-x^\ve_m\|>\delta\}=
$$
$$
= \{\sup_{t\in[0,\tau]} \|x^\ve(\cdot,t)-x^\ve_m(
\cdot,t)\|_{{L}_2(\mu)}>\delta\}\cup \{\sup_{t\in(\tau,1]}
\|x^\ve(\cdot,t)-x^\ve_m(\cdot,t)\|_{{L}_2(\mu)}>\delta\}\subset
$$
$$
\subset \{\sup_{t\in[0,\tau]} \|x^\ve(\cdot,t)-x^\ve_m(
\cdot,t)\|_{{L}_2(\mu)}>\delta\}\cup \{\tau<1\},
$$
the lemma will be proved as soon as we show that for all
$\rho>0,$
$$
\lim_{m\to\infty}\mathop{\varlimsup}\limits_{\ve\to0}\ve\log
 P
\{\sup_{t\in[0;1]}
\|x^\ve_m(\cdot,t)-x^\ve_m(\cdot,\frac{[mt]}{m})\|_{{L}_2(\mu)}
\geq\rho \}=-\infty.
$$
It follows from Chebychev's inequality and Lemma \ref{lem:2.3} that
$$
 P
\{\sup_{t\in[0;1]}
\|x^\ve_m(\cdot,t)-x^\ve_m(\cdot,\frac{[mt]}{m})\|_{{L}_2(\mu)}
\geq\rho \}\leq
$$
$$
\leq m\max_{0\leq i\leq m-1} P\{ \sup_{t\in\left[\frac{i}{m},
\frac{i+1}{m}\right]} \frac{m}{\ve}\|x^\ve_m(
\cdot,t)-x^\ve_m\left(\cdot,\frac{i}{m}\right)\|^2 _{{L}_2(\mu)}
\geq\frac{m\rho^2}{\ve}\}\leq mCe^{-\frac{m\rho^2}{\ve}}.
$$
The lemma is proved.
\end{proof}

Lemma \ref{lem:2.1}  and Lemma \ref{lem:2.4}  give us possibility to
receive the LDP for $\{x^\ve\}$   using the following theorem from
[9].
\begin{thm}
\label{thm:2.2} Let $(\mfX, \rho), \ (\cY, \sigma)$  be Polish
spaces, $\{y^\ve\}$  satisfies the LDP with  rate function $I,$
$G_m: \cY\to\mfX, m\geq1,$  are continuous functions. Assume that
there exists  $G: \cY\to\mfX$   such that for every $\alpha<\infty,$
$$
\lim_{m\to\infty} \sup_{\{y: I(y)\leq\alpha\}} \rho(G_m(y), G(y))=0.
$$
Then any family $\{x^\ve\}$   for which $\{G_m(y^\ve)\}$   is
exponentially good approximation satisfies the LDP in $\mfX$  with
rate function $$I'(x)=\inf\{I(y): x=G(y)\}.$$
\end{thm}

 Define
$$
H_1=\{h\in C([0;1]; {L}_2(\mu)): h(u,t)=
$$
$$
= u+\int^t_0\int_{\mbR}\vf(h(u,s)-p)a(p,s)dpds,
a\in{L}_2(\mbR\times[0;1]\}.
$$
\begin{thm}
\label{thm:2.3}
The family $\{x^\ve\}$  satisfies the LDP in $\mfX$  with rate
function
$$
I(h)=\begin{cases} \frac{1}{4\pi}\int^1_0\int_{\mbR} \bigg( \frac{
\cF(\dot{h}(h^{-1}( \cdot,s),s))(\lambda) }{\cF(\vf)(\lambda)}
\bigg)^2d\lambda ds, h\in H_1,\\
\infty, \ h\notin H_1.
\end{cases}
$$
\end{thm}
\begin{proof}
Define the map $G_m$   as $h=G_m(g), \ g\in\cY,$  where
$$
h(u,t)=h\left(u,\frac{k}{m}\right)+\int_{\mbR} \psi_1
\left(h\left(u,\frac{k}{m}\right)-q\right)
\left[g(q,t)-g\left(q,\frac{k}{m}\right)\right]dq,
$$
$$
t\in
\left[
\frac{k}{m},
\frac{k+1}{m}\right],
k=0,1,\ldots, m-1,
$$
$$
  h(0,u)=u.
$$
Now, observe that $G_m$  is continuous by Lemma \ref{lem:2.1}  and that
$x^\ve_m=G_m(y^\ve)$ where $y^\ve$ was defined in \eqref{2.9'}.  It is enough
to define $G$ on
\newline $ H_2=\{h\in C([0;1];
L_2(\mu)):$
$ h(u,t)=u+\int^t_0\int_{\mbR}\psi_2(u-p)a(p,s)dpds\}.
$

 Let $f=G(g)$ be the unique solution of the integral equation
$$
f(t,u)=u+\int^t_0\int_{\mbR}\psi_1(f(u,s)-q)\dot{g}(q,s)dqds.
$$
In view of Lemma \ref{lem:2.4}   the proof of the theorem is
completed by combining Theorem \ref{thm:2.1}  and Theorem
\ref{thm:2.2}, as soon as we show that for every $\alpha<\infty,$
\begin{equation}
\label{eq:29}
\lim_{m\to\infty}
\sup_{\{g|I_y(g)\leq\alpha\}} \|G_m(g)-G(g)\|=0.
\end{equation}
To this end, fix $\alpha<\infty$  and $g\in H_2$   such that
$I_y(g)\leq\alpha.$  For $g\in H_2,$   there exists
$a\in{L}_2(\mbR\times[0, 1])$  such that
$$
\dot{g}(u,t)=\int_{\mbR}\psi_2(u-p)a(p,t)dp.
$$
Since $\vf=\psi_1*\psi_2$ and $\|\vf\|_{{L}_2(\mbR)}=1,$  it follows
from the Cauchy--Schwarz inequality that for all $t\in[0, 1]$
$$
\bigg\|h(\cdot,t)-h\left(\cdot,\frac{[tm]}{m}\right)\bigg\|^2_{{L}_2(\mu)}=
$$
$$
= \int_{\mbR} \bigg( \int^t_{\frac{[tm]}{m}} \int_{\mbR}
\vf\left(h\left(u,\frac{[tm]}{m}\right)-p\right) a(p,s)dpds
\bigg)^2\mu(du)\leq
$$
$$
\leq \frac{1}{m} \int_{\mbR} \int^t_{\frac{[tm]}{m}} \int_{\mbR}
\vf^2\left(h\left(u,\frac{[tm]}{m}\right)-p\right)dp
\int_{\mbR}a^2(p,s)dpds\mu(du)=
$$
$$
=\frac{1}{m} \int^t_{\frac{[tm]}{m}} \int_{\mbR}
a^2(p,s)dpds\leq\frac{2\alpha}{m}.
$$
The Cauchy--Schwarz inequality  and the
 Lipschitz
continuity of $\vf$ imply  that for all $t\in[0;1]$ we have
$$
\|h(\cdot,t)-f(\cdot,t)\|^2_{{L}_2(\mu)}=
$$
$$
= \int_{\mbR} \bigg( \int^t_0 \int_{\mbR} \left[
\vf\left(h\left(u,\frac{[sm]}{m}\right)-p\right) -
\vf(f(u,s)-p)\right] a(p,s)dpds\bigg)^2\mu(du)\leq
$$
$$
\leq t{L}\int^t_0\int_{\mbR}a^2(p,s)dp
\|f(\cdot,s)-h\left(\cdot,\frac{[sm]}{m}\right)\|^2_{{L}_2(\mu)}ds\leq
$$
$$
\leq t{L}\int^t_0\int_{\mbR} a^2(p,s)dp \left[\|f(\cdot,s)-h(
\cdot,s)\|^2_{{L}_2(\mu)}+\frac{2\alpha}{m}\right]ds \leq
$$
$$
\leq \frac{4{L}\alpha^2}{m}+{L}\int^t_0\int_{\mbR} a^2(p,s)dp
\|f(\cdot,s)-h(\cdot,s)\|^2_{{L}_2(\mu)}ds.
$$
Hence, by Gronwall's lemma,
$$
\|f(\cdot,t)-h(\cdot,t)\|^2_{{L}_2(\mu)}\leq \frac{4{L}\alpha^2}{m}
e^{{L}\int^t_0\int_{\mbR}a^2(p,s)dpds}\leq \frac{4{L}\alpha^2}{m}
e^{2{L}\alpha},
$$
which establishes \eqref{eq:29}  and proves the theorem.
\end{proof}

{\it Remark}. Note that the distribution of $x^\ve$   can be
obtained not only by the phase scaling but  by time changing also.
Namely, denote $\wt{x}^\ve(u,t)=x(u, \ve t),$ $u\in\mbR, t\in[0;1],$
where $x$  is defined by (1.1).

\begin{lem}
\label{lem2.5}
 $x^\ve$  is equal in distribution to $\wt{x}^\ve.$
\end{lem}
\begin{proof}
Consider the $n$-point motions of both flows:
$(x^\ve(u_1,\cdot), x^\ve(u_2,\cdot), \ldots, $
\newline$x^\ve(u_n,\cdot))$
and
$(\wt{x}^\ve(u_1,\cdot), \wt{x}^\ve(u_2,\cdot), \ldots, \wt{x}^\ve(u_n,
\cdot)).$   They are diffusion processes in $\mbR$ with zero drift and diffusion
matrices
$A=(a_{ij})^n_{i,j=1}$  and
$\wt{A}=(\wt{a}_{ij})^n_{i,j=1}.$

For the first process we have
$$
a_{ij}(\ov{u})=
$$
$$
=
\lim_{t\to0}\frac{1}{t}E_{\ov{u}}\ve
\int^t_0\int_{\mbR}
\vf(x^\ve(u_i, s)-p)W(dp,ds)
\int^t_0\int_{\mbR}\vf(x^\ve(u_j, s)-p)W(dp,ds)=
$$
$$
=
\ve
\lim_{t\to0}
\frac{1}{t}
\int^t_0\int_{\mbR}
E_{\ov{u}}
\vf(x^\ve(u_i, s)-p)\vf(x^\ve(u_j, s)-p)dpds=
$$
$$
=
\ve\int_{\mbR}\vf(u_i-p)\vf(u_j-p)dp, \ \ov{u}=(u_1,\ldots, u_n)\in\mbR^n.
$$
For the second one we can write
$$
\wt{a}_{ij}(\ov{u})=
$$
$$
=
\lim_{t\to0}\frac{1}{t}E_{\ov{u}}
\int^{\ve t}_0\int_{\mbR}
\vf(x(u_i, s)-p)W(dp,ds)
\int^{\ve t}_0\int_{\mbR}\vf(x(u_j, s)-p)W(dp,ds)=
$$
$$
=
\ve
\lim_{t\to0}
\frac{1}{t}
\int^t_0\int_{\mbR}
E_{\ov{u}}
\vf(x(u_i, \ve s)-p)\vf(x(u_j, \ve s)-p)dpds=
$$
$$
=
\ve\int_{\mbR}\vf(u_i-p)\vf(u_j-p)dp, \ \ov{u}=(u_1,\ldots, u_n)\in\mbR^n.
$$

Thus, the $n$-point motions of $x^\ve$  and $\wt{x}^\ve$   are equal in
distribution. Therefore $x^\ve$ is equal in distribution to $\wt{x}^\ve.$

\end{proof}

\section{Arratia's flow and related stochastic calculus}
\label{section3}
This section is devoted to Arratia's flow of coalescing Brownian
particles. As it was mentioned in the first section, Arratia's flow can
be considered as limit case of flows with the smooth
correlation. It consists of Brownian particles, which move independently up
to the meeting then stick and move together. In contrast with the smooth
correlation case Arratia's flow can not be described by any stochastic differential
equation with a Gaussian noise. So the LDP for Arratia's flow
should be described using its intrinsic properties. In this section we
present some facts about Arratia's flow, which will be useful later on. Arratia's flow can be defined by different ways
[1, 2, 5].
We will use the following definition.
\begin{defn}
\label{defn:3.1}
Arratia's flow $\{x(u,t); u\in\mbR, t\in[0;1]\}$
is a random field with the properties:

1) for every $u\in\mbR$ \ $\{x(u,t), t\in[0;1]\}$ is a  Wiener
martingale with respect to a common filtration and $x(u,0)=u,$

2)  for every $u_1\leq u_2, \ t\in[0;1]$
$$
x(u_1, t)\leq x(u_2, t),
$$

3) $\lg x(u_1,\cdot), x(u_2,\cdot)\rg_t=(t-\tau),  \ t>\tau, $  \
$\tau=\inf\{s: x(u_1,s)=x(u_2,s)\}.$

\end{defn}

It was proved in
[13],
 that $x$  has a modification, which is a
c\`adl\`ag Markov process in $C([0; 1])$
($u\in\mbR$ now plays the role of time). Further we will consider such
modification.

We will establish the LDP for the family of random flows $\{x^\ve;
\ve>0\},$ which is built from $x$  using  the time-change. Define
for $\ve\in(0;1]$
$$
x^\ve(u,t)=x(u, \ve t), \ u\in\mbR, \ t\in[0; 1].
$$

To describe rate function we need the fundamental fact about
Arratia's flow.
This property can be formulated as follows.

\begin{lem}
\label{lem:3.1} {\rm[1, 2, 14].}
For every interval $[a; b]$  and positive time $t$
the set $\{x(u,t); u\in[a;b]\}$  is finite with probability one.
\end{lem}

In  [14] this fact was obtained as a consequence of
a more general statement about the finiteness of the  total time of free
motion  of particles in
 Arratia's flow. Consider a partition
$\lambda$  of the interval $[a; b]: a=u_0<\ldots<u_n=b.$  As usual
denote $|\lambda|=\max\limits_{k=0,\ldots, n-1}u_{k+1}-u_k.$   For
$k=1,\ldots, n$  define the random time
$$
\tau(u_k)=\inf\{t: x(u_k, t)=x(u_{k-1}, t)\}\wedge1.
$$
For $k=0$  put $\tau(u_0)=1.$   The following statement was proved in
[14, 15].

\begin{thm}
\label{thm:3.1} {\rm [14, 15].}
There exists a random variable
$$
\Gamma=\sup_\lambda\sum^n_{k=0}\tau(u_k).
$$
Here supremum means, that for arbitrary $\lambda$
\begin{equation}
\label{eq:3.2}
\Gamma\geq\sum^n_{k=0}\tau(u_k)
\end{equation}
and for arbitrary random variable $\zeta$ with property
\eqref{eq:3.2}  the following inequality holds
$$
\Gamma\leq\zeta.
$$
Moreover, $\Gamma$   can be obtained as a limit a.s.
$$
\Gamma=\lim_{|\lambda|\to0}\sum^n_{k=0}\tau(u_k).
$$
\end{thm}

Using Theorem \ref{thm:3.1} the following stochastic
integrals were built for a bounded measurable function $\vf$
\begin{equation}
\label{eq:3.3}\
\begin{split}
&
\int^b_a\int^{\tau(u)}_0\vf(x(u,s))ds=P\mbox{-}
\lim_{|\lambda|\to0}\sum^n_{k=0}\int^{\tau(u_k)}_0\vf(x(u,s))ds,\\
&
\int^b_a\int^{\tau(u)}_0\vf(x(u,s))dx(u,s)=L_2\mbox{-}
\lim_{|\lambda|\to0}\sum^n_{k=0}\int^{\tau(u_k)}_0\vf(x(u,s))dx(u,s).
\end{split}
\end{equation}

Note that the left-hand side in \eqref{eq:3.3}  contains two
symbols of integral and only one symbol of differential. It
emphasizes that the second differential can be substituted by
$\tau(u)$   which formally possesses the property
$$
\sum_{u\in[a,b]}\tau(u)<+\infty.
$$
The integrals from \eqref{eq:3.3}  allow to formulate the Girsanov
theorem for Arratia's flow [14]. If we consider the flow $y,$   which
is built analogously to Arratia's flow but instead of a Wiener process
a diffusion process with the drift $\vf$  and variance 1 is used,
then it can be proved [14], that in an appropriate functional space
the distribution $P_y$  of this flow is absolutely continuous with
respect to the distribution $P_x$  of Arratia's flow and
\begin{equation}
\label{eq:3.4}
\frac{dP_y}{dP_x}=\exp
\bigg\{
\int^b_a\int^{\tau(u)}_0\vf(x(u, s))dx(u, s)-
\frac{1}{2}\int^b_a\int^{\tau(u)}_0\vf^2(x(u, s))ds\bigg\}.
\end{equation}
Note that the form   of the derivative  in \eqref{eq:3.4} is very
natural. It consists of the sum of usual terms for Girsanov theorem
along the pieces of trajectories of particles in Arratia's flow up to
the moment of the first meeting. We will prove, that rate
function for $\{x^\ve\}$  is the infinite sum of  rate
functions for the Wiener process $x(u,\cdot)-u, u\in\mbR$   up to the
moment of the meeting. The main result will be proved in two steps.
In the next two sections we will consider the case of finite number of
 particles and the general case will be treated in the last
section.

\section{LDP for stopped Wiener process}
\label{section4}
Here we consider a Wiener process $\vec{w}$  in $\mbR^d$  starting
from a point $\vec{u}.$  Let $B\subset\mbR^d$  be a closed set.
Define the stopping time
$$
\tau=\inf\{t: \vec{w}(t)\in B\}\wedge1.
$$
Consider in the space $C([0;1], \mbR^d)$  the family of random
elements $\{\vec{y}\,^\ve; \ve\in(0; 1]\}$  defined as follows
$$
\vec{y}\,^\ve(t)=\vec{\eta}(\ve t), \ t\in[0; 1],
$$
where $\vec{\eta}(t)=\vec{w}(t\wedge\tau).$  To describe rate
function for the family $\{\vec{y}\,^\ve\}$  let us denote for
$\vec{f}\in C([0; 1], \mbR^d)$
$$
\tau(\vec{f})=\inf\{t: \ \vec{f}(t)\in B\}\wedge1.
$$
Define
$$
\Phi(\vec{f})(t)=\vec{f}(t\wedge\tau(\vec{f})), \ t\in[0; 1].
$$
Denote by $H$  the subset of $C([0; 1], \mbR^d)$
consisting of functions with  square-integrable derivative.
Now define the rate function $I$  on $C([0; 1], \mbR^d)$  as follows
$$
I(\vec{g})=
\begin{cases}
\frac{1}{2}\int^1_0\|\dot{\vec{g}}(t)\|^2dt, \ \vec{g}(0)=\vec{u}, \
\vec{g}\in H\cap \Phi(C([0;1], \mbR^d)),\\
+\infty, \ \vec{g}(0)\ne\vec{u} \ \mbox{or} \ \vec{g}\notin H\cap\Phi(C([0;1], \mbR^d)).
\end{cases}
$$
The following statement is the main result of this section.
\begin{thm}
\label{thm:4.1}
The family $\{\vec{y}\,^\ve\}$   satisfies the LDP in
$C([0;1], \mbR^d)$  with rate function $I.$
\end{thm}

Before proving the theorem let us note that it can not be obtained
directly from the LDP for Wiener process by contraction principle or
its modification [9], since the map $\Phi$  is not continuous. Moreover,
in some cases the set of discontinuities for $\Phi$  has a
positive Wiener measure. The following example shows such possibility.
\begin{expl}
\label{expl:4.1}
Let $d=2$  and $B$  be a Sierpinski carpet of the positive Lebesgue measure,
which is built by the usual procedure [16],
$$
B=\cap^\infty_{n=1}B_n.
$$
Here for every $n\geq1$ \ $B_n$  is a finite union of disjoint
closed squares, $B_1=[0;1]^2,$ $B_{n+1}\subset B^\circ_n, n\geq1.$
The squares in $B_n$  have equal sides $d_n$ and $d_n\to0,
n\to\infty.$ By the Kakutany criterium the process $\vec{w}$
visits $B$  with a positive probability, i.e.
\begin{equation}
\label{eq:4.1}
P\{\exists \ t\in[0;1): \ \vec{w}(t)\in B\}>0.
\end{equation}
Now consider the subset $A\subset C([0; 1], \mbR^d)$
consisting of functions with the properties
$$
\tau(\vec{f})<1, \ \Phi(\vec{f})\ne\vec{f}.
$$
It can be proved, that $A$ is a Borel set. It follows from
\eqref{eq:4.1}, that $A$ has a positive Wiener measure. Now check,
that every point of $A$ is a point of discontinuity for $\Phi.$  Let
$\vec{f}\in A.$  For every $n\geq1$  define $\Delta_n=\{t\in[0;1]:
\vec{f}(t)\in B^\circ_n\}.$ Then $\Delta_n$  is a union of
disjoint open intervals (possibly with one interval of the type
$(\alpha; 1]$). Consider one interval $(\alpha;\beta)$  from
$\Delta_n.$ Then $\vec{f}([\alpha;\beta])$  is a subset of one of
the squares which form $B_n.$  Denote this square by $J.$   Then
$\vec{f}(\alpha), \vec{f}(\beta)\in\partial J.$  Define the function
$\vec{f}_n$  on $[\alpha;\beta]$  in such a way that it is continuous
and $\vec{f}_n([\alpha;\beta])\subset\partial J,$
$\vec{f}_n(\alpha)=\vec{f}(\alpha),$
$\vec{f}_n(\beta)=\vec{f}(\beta).$  If we proceed in the same way on
every interval from $\Delta_n$   and define for $t\notin\Delta_n$
$\vec{f}_n(t)=\vec{f}(t),$  then we will get $\vec{f}_n\in C([0;1],
\mbR^2).$  This function has the following properties:

1) $\max\limits_{t\in[0;1]}\|\vec{f}_n(t)-\vec{f}(t)\|\leq2d_n\to0,
n\to\infty,$

2) $\tau(\vec{f}_n)=1, \ n\geq1.$

Consequently,
$$
\Phi(\vec{f}_n)=\vec{f}_n, \ n\geq1,
$$
and
$$
\Phi(\vec{f}_n)\to\vec{f}, \ n\to\infty.
$$
But $\vec{f}\ne\Phi(\vec{f}).$  Hence $\vec{f}$ is a point of
discontinuity for $\Phi.$  Note that the crucial point in this
example is the possibility to approximate uniformly arbitrary Wiener
trajectory by continuous functions,that do not take values in the set $B.$
\end{expl}

\begin{proof}[Proof of Theorem \ref{thm:4.1}]
Let $G$ be an open set in $C([0;1],\mbR^d).$  Denote by $I_0$ the rate
function for the family
$$
\vec{w}^\ve(t)=\vec{w}(\ve t), \ t\in[0;1], \ \ve\in(0;1],
$$
i.e.
$$
I_0(\vec{g})=
\begin{cases}
\frac{1}{2}\int^1_0\|\dot{\vec{g}}(t)\|^2dt,& \ \vec{g}(0)=\vec{u},
\ \vec{g}\in H,\\
+\infty,&  \   \vec{g}(0)\ne\vec{u}  \ \mbox{or} \ \vec{g}\notin H.
\end{cases}
$$
Note that
$$
\vec{y}\,^\ve=\Phi(\vec{w}_\ve), \ \ve\in(0; 1].
$$
We will prove, that
\begin{equation}
\label{eq:4.2}
\inf_GI=\begin{cases}
\inf_{\Phi^{-1}(G)^\circ}I_0, & \  \ \Phi^{-1}(G)^\circ\ne\O,\\
+\infty, &\  \ \Phi^{-1}(G)^\circ=\O.
\end{cases}
\end{equation}
Here, as usual, $A^\circ$   denote the interior  of the set $A.$
Consider first the  case when $\Phi^{-1}(G)^\circ=\O.$   Let
$\vec{f}\in G.$ Note that if $\tau(\vec{f})=1,$  then there exists
$\vec{g}\in G$  close enough to $\vec{f}$  such that
$\vec{g}([0;1])\cap B=\O.$ The function $\vec{g}$  can be defined as
$$
\vec{g}(t)=\vec{f}(t\wedge1-\delta)
$$
for small enough $\delta.$  Then there exists $\sigma>0$  such, that the
open ball $B(\vec{g}, \sigma)$   with the center $\vec{g}$ and radius
$\sigma$ is subset of $G$   and has the property
$$
\forall \ \vec{h}\in B(\vec{g}, \sigma): \ \vec{h}([0; 1])\cap B=\O.
$$
Hence
$$
\forall \ \vec{h}\in B(\vec{g}, \sigma): \ \Phi(\vec{h})=\vec{h}.
$$
Consequently, $B(\vec{g},\sigma)\subset\Phi^{-1}(G)$  which
contradicts our supposition. It follows from the previous
considerations, that now
$$
\forall \ \vec{f}\in G: \tau(\vec{f})<1.
$$
Consider $\vec{f}\in G$    such, that
$$
\Phi^{-1}(\vec{f})=\O.
$$
For such $\vec{f}$ by definition
$I(\vec{f})=+\infty.$   Now let $\Phi^{-1}(\vec{f})\ne\O.$  Then
$\vec{f}=\Phi(\vec{f}).$   Note that the open set $G$  contains with the
function $\vec{f}$ every function of the kind
\begin{equation}
\label{eq:4.3}
\vec{h}(t)=\vec{f}(t\wedge\tau(\vec{f})-\delta)
\end{equation}
for small enough $\delta.$  But for such $\vec{h}$ $\tau(\vec{h})=1$
which again contradicts our supposition $\Phi^{-1}(G)^\circ=\O.$

It remains to consider the case $\Phi^{-1}(G)^\circ\ne\O.$ Let
$\vec{f}$  belongs to $G.$   If $\tau(f)=1$   and $I(f)<+\infty,$
then, similarly to the previous considerations, there exists a
sequence $\{\vec{h}_n; n\geq1\}$   from $\Phi^{-1}(G)^\circ$ such,
that $\vec{h}_n\to\vec{f},$ $I(\vec{h}_n)\to I(\vec{f}),$
$n\to\infty,$   and
$$
\forall \ n\geq1: \ \vec{h}_n([0;1])\cap B=\O.
$$
Since for such functions
$$
I(\vec{h}_n)=I_0(\vec{h}_n),
$$
$$
\Phi(\vec{h}_n)=\vec{h}_n,
$$
then
\begin{equation}
\label{eq:4.4} I(\vec{f})\geq\inf_{\Phi^{-1}(G)^\circ}I_0.
\end{equation}
Now take $\vec{f}\in G$   such, that $\tau(\vec{f})<1$  and
$\Phi^{-1}(\vec{f})=\O,$ $I(\vec{f})<+\infty.$ Using the
approximation like \eqref{eq:4.3} we again can get inequality
\eqref{eq:4.4}. This completes the proof of relation
\eqref{eq:4.2}.

Now for open set $G$
$$
\mathop{\varliminf}\limits_{\ve\to0}\ve\ln P\{y_\ve\in G\}=
\mathop{\varliminf}\limits_{\ve\to0}\ve\ln P\{\Phi(w_\ve)\in G\}=
$$
$$
= \mathop{\varliminf}\limits_{\ve\to0}\ve\ln
P\{w_\ve\in\Phi^{-1}(G)\}\geq
\mathop{\varliminf}\limits_{\ve\to0}\ve\ln
P\{w_\ve\in\Phi^{-1}(G)^0\}\geq
$$
$$
\geq\inf_{\Phi^{-1}(G)^\circ} I_0=\inf_GI.
$$
Here we use the LDP for Wiener process.

Consider the closed set $F\subset C([0;1], \mbR^d).$   Let us prove for
$F$ an analog of \eqref{eq:4.2}.   Namely,

\begin{equation}
\label{eq3.7}
\inf_FI=\inf_{\ov{\Phi^{-1}(F)}} I_0,
\end{equation}
where $\bar{A}$  denote the closure of $A.$    To check \eqref{eq3.7}
take a function $\vec{f}\in\ov{\Phi^{-1}(F)}.$ Suppose, that
$I_0(\vec{f})<+\infty.$

If $\tau(\vec{f})=1,$  then $\Phi(\vec{f})=\vec{f}.$   From other side
$$
\vec{f}=\lim_{n\to\infty}f_n,
$$
where
$$
\vec{f}_n\in\Phi^{-1}(F), \ n\geq1.
$$
Consequently,
$$
\tau(\vec{f}_n)\to1, \ n\to\infty,
$$
$$
\Phi(\vec{f}_n)\to\vec{f}, \ n\to\infty.
$$
Hence $\vec{f}\in\Phi^{-1}(F)$  and by definition
$I_0(\vec{f})=I(\Phi(f)).$

Now, let $\tau(\vec{f})<1.$ Then condition $I(\vec{f})<+\infty$  implies
that $\Phi(\vec{f})=\vec{f}.$  Consider a sequence
$\{\vec{f}_n; n\geq1\}$  from $\Phi^{-1}(F)$   such, that
$$
\vec{f}_n\to\vec{f}, \ n\to\infty.
$$
Then the sequence $\{\Phi(\vec{f}_n); n\geq1\}$   contains
subsequence converging to a certain element $\vec{g}$ of $F.$   It can
be easily verified, that $\vec{f}=\vec{g}.$  Hence $\vec{f}\in F$   and
$\vec{f}\in\Phi^{-1}(F).$   This proves \eqref{eq3.7}. Now
$$
\mathop{\varlimsup}\limits_{\ve\to0+}\ve\ln P\{\vec{y}\,^\ve\in F\}=
\mathop{\varlimsup}\limits_{\ve\to0+}\ve\ln P\{\vec{w}^\ve
\in\Phi^{-1}(F)\}\leq
$$
$$
\leq \mathop{\varlimsup}\limits_{\ve\to0+}\ve\ln P\{\vec{w}^\ve
\in\ov{\Phi^{-1}(F)}\}\leq -\inf_{\ov{\Phi^{-1}(F)}}I_0=-\inf_FI.
$$
The theorem is proved.
\end{proof}

\section
{ LDP for finite-dimensional distributions of Arratia's flow}  In this
section we will consider the LDP for process $\vec{x}\,^\ve=\{x(u_1,
\ve t), \ldots ,$\newline$x(u_n, \ve t), $ $ t\in[0; 1]\}, \ve\in(0;
1]$ where $x$ is Arratia's flow, $u_1<\ldots <u_n$  are fixed points.
Denote by $C_{\vec{u}}([0; 1], \mbR^n)$   the subset of $C([0; 1],
\mbR^n)$ consisting of functions with the property
$f_i(0)=u_i, i=1,\ldots, n. $   Consider the moments of sticking
$\tau_1\leq\ldots\leq \tau_{n-1}$   (possibly from a certain
number they are equal to 1). Using these moments one can define the
map $\Phi$  on $C_{\vec{u}}([0; 1], \mbR^n)$ similarly
as it was done in the
previous section. Namely, after the moment of meeting of some
coordinates put all of them to be equal to the coordinate with the
smallest number. Then $\vec{x}=\Phi(\vec{w}),$   where $\vec{w}$  is
a standard Wiener process starting from $(u_1,\ldots, u_n).$
Exactly as in the previous section one can prove the following
theorem.
\begin{thm}
\label{thm5.1v} $\{\vec{x}\,^\ve\}$  satisfies the  LDP  with rate
function $I.$ Here $I(\vec{f})$  is equal to $+\infty$ if
$\Phi(\vec{f})\ne\vec{f}$ or if
 $\vec{f}$   has not a square integrable
derivative. In the opposite case
$$
I(\vec{f})=\frac{1}{2}\sum^n_{k=0}\int^{\tau(u_k)}_0f'_k(s)^2ds.
$$
\end{thm}

Here, as in section 3,
$$
\begin{aligned}
&
\tau(u_0)=1,\\
& \tau(u_k)=\inf\{t: f_k(t)=f_{k-1}(t)\}\wedge1, \ k=1, \ldots, n.
\end{aligned}
$$

\section{  LDP for Arratia's flow on finite interval}  Consider Arratia's
 flow $\{x(u, t); u\in[0; 1], t\in[0; 1]\}.$
Following [14]  we suppose, that $x$  is already modified to be a
c\`adl\`ag process with respect to the variable $u$  with values in
$C([0; 1]).$  Denote by $\lambda$  the Lebesgue measure on $[0; 1]$
and for every $t\in[0; 1]$ define the random measure $\mu_t$  as the
image
$$
\mu_t=\lambda\circ x(\cdot, t)^{-1}.
$$
As it was proved in [2, 14], $\mu_t$   for every $t>0$  is a
random measure concentrated in a finite number of points. Note that $x$   can be fully recovered from $\{\mu_t; t\in(0; 1]\}.$  For
probability measures on $\mbR$  we will use the L\'evy--Prokhorov
distance $\sigma$ [17], which metrizes weak convergence.
For $\ve\in(0; 1]$  define
$$
\mu^\ve_t=\mu_{\ve t}, \ t\in[0; 1].
$$
We will establish the LDP for the processes $\mu^\ve.$ To do this
consider the sequence of the $(2^n+1)$-point motions from Arratia's flow.
For every $n\geq1$  define $x_{n\ve}$  as the family
$\left\{x\left(\frac{k}{2^n}, \ve t\right); k=0,\ldots, 2^n, t\in[0;
1]\right\}.$   Recall $x_{n\ve}$  satisfies the LDP in $C([0; 1],
\mbR^{2^n+1})$   accordingly to the previous section. Let us
consider the sequence $\{x_{n\ve}; n\geq1\}$  as an element of the
product $ \prod^\infty_{n=1}C([0; 1], $\newline $\mbR^{2^n+1}) $ which is
equipped with the product topology. Note that for every $k\leq n$ \
$x_{k\ve}$  is a continuous function of $x_{n\ve}.$ Consequently, by
the contraction principle, the vector $(x_{1\ve}, \ldots, x_{n\ve})$
satisfies the LDP in $ \prod^n_{k=1}C([0; 1], \mbR^{2^k+1}) $ with
rate function $I_n$  (the same as for $x_{n\ve}$). Applying Dawson
and G\"{a}rtner theorem about random sequences [18]  one can get
that $\{x_{n\ve}; n\geq1\}$   satisfies the  LDP in $
\prod^\infty_{n=1}C([0; 1], \mbR^{2^n+1}) $ with rate function
\begin{equation}
\label{eq6.1}
I(\{f_n; n\geq1\})=\sup_{n\geq1}I_n(f_n).
\end{equation}
The last expression can be rewritten as follows. Consider a function
$f$ which is defined on the set
$\mbQ_2=\left\{\frac{k}{2^n}; k=0,\ldots, 2^n, n\geq1\right\},$
takes values in $C([0; 1])$ and satisfies the condition
$f(r,0)=r, r\in\mbQ_2.$
 Such functions are in one to one
correspondence with the sequences $\{f_n; n\geq1\}$   from
\eqref{eq6.1}. If $I(f)<+\infty,$  then for arbitrary $r_1\leq r_2$
from $\mbQ_2$ and $t\in[0; 1]$
\begin{equation}
\label{eq6.2}
f(r_1, t)\leq f(r_2, t).
\end{equation}
For $f$ satisfing \eqref{eq6.2}  and $r_1<r_2$  from $\mbQ_2$
define
$$
\tau(r_1, r_2)=\inf\{t: f(r_1, t)=f(r_2, t)\}\wedge 1.
$$
Then for every $n\geq1$
$$
I_n(f)=\frac{1}{2}\sum^{2^n}_{k=0}\int_0^
{\tau\left(\frac{k-1}{2^n}, \frac{k}{2^n}\right)} {\dot
f}\left(\frac{k}{2^n}, s\right)^2ds,
$$
where
$$
\tau\left(\frac{-1}{2^n}, 0\right)=1.
$$
Note that $\{I_n(f); n\geq1\}$  is nondecreasing and
$$
I(f)=\lim_{n\to\infty}I_n(f).
$$
Define the space of
trajectories for Arratia's flow. Denote by $\cM$ the set of real-valued
functions defined on $[0; 1]^2,$ which have the following properties:

1) for every $u\in[0; 1]$  $y(u, \cdot)\in C([0; 1]),$

2) for all $u_1\leq u_2, t\in[0; 1]$
$$
y(u_1, t)\leq y(u_2, t),
$$

3) for every $t\in[0; 1]$  \ $y(\cdot, t)$  is a c\`adl\`ag function,

4) for all $u_1, u_2\in[0; 1]$
$$
y(u_1, t)=y(u_2, t), \ t\geq\tau_{u_1u_2},
$$
$$
\tau_{u_1u_2}=\inf\{s: y(u_1, s)=y(u_2, s)\},
$$

5) for every $u\in[0; 1]$
$$
y(u,0)=u.
$$

An arbitrary element of $\cM$  can be treated as a continuum forest. We will
endow $\cM$  with the distance
$$
\rho(y_1, y_2)=\max_{t\in[0; 1]}\sigma(y_1(\cdot, t), y_2(\cdot,
t)),
$$
where $\sigma$ is the L\'evy--Prokhorov distance between c\`adl\`ag
functions. Note that the convergence of $y$ in the distance
$\sigma$  for a fixed $t$ is equivalent to the weak convergence of
$\mu_t.$   Now define the subset $\cR$ of $C([0; 1])^\infty$   as
follows. $\cR$ consists of real-valued functions defined on
$\mbQ_2\times[0; 1],$ which have the properties:

1) for every $r\in\mbQ_2$
$$
y(r, \cdot)\in C([0; 1]),
$$

2) for all $r_1\leq r_2, t\in[0; 1]$
$$
y(r_1, t)\leq y(r_2, t),
$$

3) for every $r\in\mbQ_2$
$$
y(r,0)=r.
$$

Note that $\cR$  as a subset of $C([0; 1])^\infty$ is closed in the
distance of pointwise  convergence on $\mbQ_2$ and  uniform convergence on $[0; 1].$

Consider a map $i: \cR\to\cM$ which is defined as follows\newline
$\forall \ u,t\in[0; 1], \ y\in\cR:$
$$
i(y)=\wt{y}, \ \wt{y}(u, t)=\inf_{\begin{subarray}{l}
r\in\mbQ_2\\
r>u
\end{subarray}}
y(r,t).
$$
\begin{lem}
\label{lem6.1.}
$i$ is continuous mapping.
\end{lem}
\begin{proof}
Suppose, that $y_n\to y, n\to\infty$ in $\cR.$   For a given positive $\ve$
consider a partition $0=r_0<\ldots<r_m=1$    consisting of the
points from $\mbQ_2$ and satisfying the property
$$
\max_{k=0,\ldots, m-1}(r_{k+1}-r_k)<\frac{\ve}{2}.
$$
Take $n$ such, that\newline
$\forall \ k=0, \ldots, m:$
$$
\max_{[0; 1]}|y_n(r_k, t)-y(r_k, t)|<\frac{\ve}{2}.
$$
Consider an arbitrary $u\in[0; 1).$  Then there exists
$r_k\in(u, u+\ve).$  For such $r_k$  and arbitrary $t\in[0; 1]$
$$
\wt{y}_n(u, t)\leq y_n(r_k, t)\leq y(r_k, t)+\frac{\ve}{2}\leq
$$
$$
\leq\wt{y}(u+\ve, t)+\frac{\ve}{2},
$$
and
$$
\wt{y}(u, t)\leq y(r_k, t)\leq y_n(r_k, t)+\frac{\ve}{2}\leq
\wt{y}_n(u+\ve, t)+\frac{\ve}{2}.
$$
Now take $u\in[\ve; 1].$ There  exists $r_k\in(u-\ve, u).$  For such $r_k$
and arbitrary $t\in[0; 1]$
$$
\wt{y}_n(u, t)\geq y_n(r_k, t)\geq y(r_k, t)-\frac{\ve}{2}\geq
\wt{y}(u-\ve, t)-\frac{\ve}{2},
$$
and
$$
\wt{y}(u, t)\geq y(r_k, t)\geq y_n(r_k, t)-\frac{\ve}{2}\geq
\wt{y}_n(u-\ve, t)-\frac{\ve}{2}.
$$
The lemma is proved.
\end{proof}

As a consequence of the previous lemma and the contraction principle we
obtain the  LDP for Arratia's flow.
\begin{thm}
\label{thm6.1} Let $\{x^\ve(u,t)=x(u, \ve t), u\in(0; 1), t\in[0;
1]\},$ $\ve\in(0; 1)$  are the random fields obtained from Arratia's
flow by the time changing. Then the family $\{x^\ve\}_{\ve>0}$
satisfies the  LDP in $\cM$ with the rate function
\begin{equation}
\label{eq6.3}
I(x)=\inf_{i(h)=x}I_0(h),
\end{equation}
where $I_0(h)$  is given by \eqref{eq6.1}.
\end{thm}

\end{document}